\documentclass[reqno]{amsart}
\usepackage{amssymb,amsthm,amstext}
\usepackage{amsmath}
\usepackage[utf8]{inputenc}
\usepackage{xcolor}
\usepackage{mathrsfs}  

\usepackage{mathtools}       
\mathtoolsset{showonlyrefs}

\usepackage{hyperref}

\usepackage{newcent}        

\newtheorem{theorem}{Theorem}[section]
\newtheorem{lemma}[theorem]{Lemma}
\newtheorem{proposition}[theorem]{Proposition}

\theoremstyle{remark}

\newtheorem{remark}[theorem]{Remark}

\newcommand{\mc}[1]{{\mathscr #1}}

\newcommand{\mb}[1]{{\mathbf #1}}
\newcommand{\bb}[1]{{\mathbb #1}}

\newcommand{\PP}{\mathbb{P}}

\newcommand{\EE}{\mathbb{E}}
\newcommand{\RR}{\mathbb{R}}
\newcommand{\NN}{\mathbb{N}}
\newcommand{\ZZ}{\mathbb{Z}}

\newcommand{\YY}{\mathcal Y}

\newcommand{\xxi}{\overline{\xi}}

\newcommand{\norm}[1]{\Vert #1 \Vert}

\newcommand{\ind}[1]{\textbf{1}_{#1}}

\newcommand{\lrp}[1]{\left(#1\right)}
\newcommand{\lrc}[1]{\left[#1\right]}
\newcommand{\lrch}[1]{\left\{ #1\right\} }
\newcommand{\lrb}[1]{\left\langle #1 \right\rangle}

\begin{document}

\title{Symmetric exclusion as a random environment: invariance principle}

\author{Milton Jara}
\address{Instituto de Matem\'atica Pura e Aplicada, Estrada Dona Castorina 110, 22460-320  Rio de Janeiro, Brazil.}  \email{mjara@impa.br}

\author{Ot\'avio Menezes}
\address{\noindent Centro de An\'alise Matem\'atica, Geometria e Sistemas Din\^amicos\\
	Instituto Superior T\'ecnico\\
	Av. Rovisco Pais, 1049-001 Lisboa, Portugal.}
\email{otavio.menezes@tecnico.ulisboa.pt}

\begin{abstract}
	We establish an invariance principle for a one-dimensional random walk in a dynamical random environment given by a speed-change exclusion process. The jump probabilities of the walk depend on the configuration of the exclusion in a finite box around the walker. The environment starts from equilibrium. After a suitable space-time rescaling, the random walk converges to a sum of two independent processes, a Brownian motion and a Gaussian process with stationary increments.
\end{abstract}


\maketitle

\tableofcontents

\section{Introduction}

This paper establishes an invariance principle for a family of random walks in dynamical random environments (RWDRE) on $\ZZ$ introduced in \cite{afjv}. 
In \cite{afjv}, the authors prove a law of large numbers for the random walk and for the environment as seen by the walker. 
The article \cite{ajv2017} proves the corresponding large deviations principle.
Our article completes the picture by proving an invariance principle.
We define the model in Section \ref{notation_and_results}. For now, a good picture to keep in mind is that of a random walk on top of a simple symmetric exclusion process.
The walker moves according to the following rule: after waiting an exponentially distributed random time, it flips a coin.
If the coin comes up heads, the walker jumps either to its left or its right neighbour, with the same probability; if the coin comes up tails, the walker checks if there is a particle beneath it. 
If there is, he jumps to its left neighbour, and if there is not he jumps to its right neighbour. We prove that, after a proper rescaling of time, space and waiting rates of the walker, its trajectory looks like the sum of a Brownian motion and an independent Gaussian process of stationary increments. For certain choices of the parameters, the limiting Gaussian process is a fractional Brownian motion of Hurst parameter $ \frac{3}{4} $.

Our article fits into two niches in the current probabilistic literature: random walks on dynamical random environments (RWDRE) and scaling limits of interacting particle systems.  
The symmetric exclusion in \cite{avena2009large} was introduced as an example of dynamical random environment with slowly decaying time correlations. 
This followed a series of works dealing with random walks on so-called ``fast mixing'' environments. 
These are models where, in some sense, the environment refreshes itself after a finite (but maybe random) number of jumps of the walk. 
In this setting, one expects the walk to behave as if the environment were deterministic.  That is, a law of large numbers holds, fluctuations around the limit are Gaussian and large deviation probabilities decay exponentially fast. See \cite{avenaphd} for an overview. 
Fast mixing environments are opposite, in a sense, to static environments, where the (random) transition kernel for the walk at each site does not change in time. 
In the static scenario, the walk can get trapped for a long time in small regions, leading to a rich phenomenology. For instance, it can present subdiffusive behavior and polynomial decay of the large deviation probabilities, see \cite{zeitouni2009random}.
In the fast mixing scenario, the traps dissolve before the walk can get stuck for too long. What happens in the middle? This question motivated the study of symmetric exclusion as a random environment, as well as of a couple of other conservative interacting particle systems, see \cite{avena2012continuity}, \cite{bethuelsen2016contact}, \cite{huveneers2015random}, \cite{hilario2015random}, \cite{BloHilTei2018}. The goal of these works is to prove laws of large numbers, central limit theorems and large deviation principles, and most results hold only in a subset of the space of parameters. 
Simulations reported in \cite{avena2012continuity} indicate that trapping may happen when the dynamical random environment is the one-dimensional exclusion processes, indicating that the random walk should have anomalous scaling on some region of parameters.

In the same direction, we mention the recent works \cite{cooling1} and \cite{cooling2}, that analyse a new family of random environments interpolating between static and fast mixing.

The model introduced in \cite{afjv} plays with the idea of slow mixing in a different way. Let $n$ be a scaling parameter, that will be sent to $\infty$. When the environment is given by the symmetric exclusion process, it is reasonable to introduce a diffusive space-time scaling $x \mapsto \frac{x}{n}$, $t \mapsto tn^2$. Under this scaling, the evolution of the exclusion process satisfies a law of large numbers (the so-called {\em hydrodynamic limit}) and a central limit theorem. In \cite{afjv} the exponential clock of the random walk is slowed down by a factor $\frac{\lambda}{n}$, where $\lambda>0$. Then, at least heuristically, between two jumps of the random walker the environment achieves local equilibrium in a region of size $\sqrt n$ around the walker, which is exactly the size at which fluctuations appear. Therefore, the walker should see a randomly evolving equilibrium of the environment process. This heuristics can be made rigorous by means of the formalism of hydrodynamic limits of interacting particle systems, which yields laws of large numbers \cite{afjv} and large deviation principles \cite{ajv2017}.

In this article we show a central limit theorem for the random walk under the scaling introduced in \cite
{afjv}, assuming the dynamic random environment is stationary in time. The scaling limit is then a mixture of two independent Gaussian processes: a Brownian motion and a process with stationary increments introduced in \cite{slafips} as the scaling limit of the occupation time of the origin in the weakly asymmetric exclusion process. The role of the weak asymmetry in \cite{slafips} is played here by the asymptotic speed of the random walk. When the asymptotic speed is zero, the additional Gaussian process corresponds to a fractional Brownian motion of Hurst exponent $H = 3/4$. Up to our knowledge, no previous work has been able to obtain an anomalous (superdiffusive in our case) scaling limit for a random walk in dynamical random environment.

From the hydrodynamic limits side, we compute the scaling limit of an additive functional without explicit knowledge of the invariant measures. 
On our way to obtain this result we prove an estimate on the relative entropy between the environment process at time $t$ and a product measure, using a modification of Yau's Relative Entropy method, introduced in \cite{yau1991relative}. 
This method is nowadays a standard tool for proving hydrodynamic limits. However, the current state of the art only yields a bound of order $o\big( \tfrac{t}{n}\big)$. This bound is enough to derive a law of large numbers and also a large deviations principle, but it is far from what is required in order to prove a central limit theorem.
Our main technical innovation is the derivation of a bound of order $\mc O(t)$, obtained with a different implementation of the Relative Entropy method, which is of independent interest.

Our result can be viewed as a variation on the problem of the tagged particle. 
The seminal article on this problem is \cite{kipnis1986central}, where a 
powerful method for establishing scaling limits of tagged particles was 
introduced. The method considers the \emph{environment as seen from the 
particle}, $\xi_t(x):=\eta_t(x+x_t)$ ($\eta_t$ is the particle system and $x_t$ 
is the tagged particle) and writes the position of the tagged particle as a 
martingale plus an additive functional. The 
martingale part can be handled by the Martingale Functional Central Limit 
Theorem (MFCLT), see Theorem \ref{t2}. The problem reduces, therefore, to 
studying the scaling limit of the additive functional. The work 
\cite{kipnis1986central} gives sufficient conditions to approximate this 
additive functional by a martingale, thus establishing Brownian motion as the 
scaling limit of the tagged particle. We point the reader to 
\cite{komorowski2012fluctuations} for a comprehensive exposition of the 
martingale approximation technique, and to \cite{avena2012symmetric} for an 
application in RWRE. In our model, the additive 
functional does not converge to Brownian motion, but to a singular functional of 
the density fluctuation field associated to the environment process. This 
functional turns out to be identical to the scaling limit of the occupation time 
of the origin of a stationary, weakly asymmetric exclusion process. The problem 
of the asymptotic behavior of the occupation time was already considered in the 
60's \cite{Por} in the case of independent particles and generalized to the case 
of interaction by branching, see \cite{HolStr} and the references therein. 
However, apart from dynamics which can be handled with duality techniques, the interacting case was open until \cite{slafips}. In 
this article we follow the approach of \cite{slafips}, adapted to deal with the 
lack of knowledge of the invariant measure of the environment process.

\section{Notation and results}\label{notation_and_results}

\subsection{A warm-up example}

Let $(\eta_t)_{t \geq 0}$ be the simple symmetric exclusion process (SSEP) on $\bb Z$, namely the Markov process that takes values in $\{0,1\}^{\bb Z}$ and is generated by the operator
\[
L^{ex} f (\eta) = \sum_{x \in \bb Z} \lrc{f(\eta^{x,x+1}) - f(\eta)},
\]
where $f: \{0,1\}^{\bb Z}$ is a local function and $\eta^{x,x+1}$ is obtained from $\eta$ by interchanging the values of $\eta(x)$ and $\eta(x+1)$. 
Let $\rho \in (0,1)$ and let $\nu_\rho$ denote the Bernoulli product measure in $\{0,1\}^{\bb Z}$. We assume that $\eta_0$ has law $\nu_\rho$. In that case, the law of $\eta_t$ is $\nu_\rho$ for any $t \geq 0$. 

The process $(\eta_t)_{t \geq 0}$ will serve as a dynamical random environment for a random walk that we will define now. Let $n \in \bb N$ be a scale parameter and let $\eta_t^n = \eta_{tn^2}$ be the SSEP with a diffusive speeding-up. Let $\alpha, \beta \geq 0$ be such that $\alpha+\beta >0$. Let $(x_t^n)_{t \geq 0}$ be the time-inhomogeneous chain with the following dynamics: the chain waits an exponential time of rate $n$, at the end of which it jumps to one of its two neighbors. To make its choice, it looks at the value $\eta_t^n(x)$ of the SSEP at its current location $x$. If $\eta_t^n(x)=1$, the chain jumps to the right with probability $\frac{\alpha}{\alpha+\beta}$ and to the left with probability $\frac{\beta}{\alpha+\beta}$. If $\eta_t^n(x) = 0$, the probabilities are reversed: the chain jumps to its right with probability $\frac{\beta}{\alpha+\beta}$ and to its left with probability $\frac{\alpha}{\alpha+\beta}$.

The process $(x_t^n)_{t \geq 0}$ obtained in this way is called a {\em random walk in dynamic random environment}. In \cite{afjv}, the authors proved that
\[
\lim_{n \to \infty} \frac{x_t^n}{n} = v(\rho) t,
\]
where $v = (2\rho -1 ) \frac{\alpha-\beta}{\alpha+\beta}$, that is, a law of large numbers for the random walk $(x_t^n)_{t \geq 0}$. The corresponding large deviations principle has been proved in \cite{avena2009large}. Our main goal is to prove the corresponding central limit theorem: we will prove that
\[
\lim_{n \to \infty} \frac{x_t^n - v(\rho) t n}{\sqrt n} = (\alpha+\beta) B_t + Z_t,
\]
where $(B_t)_{t \geq 0}$ is a standard Brownian motion and $(Z_t)_{t \geq 0}$ is a Gaussian process with stationary increments, independent of $(B_t)_{t \geq 0}$.

\begin{remark}
The variance of the process $(Z_t)_{t \geq 0}$ can be explicitly computed. It corresponds, modulo a proper choice of constants, to the process appearing in Theorem 6.3 of \cite{slafips}.
\end{remark}

\subsection{General setting}

Let $\Omega = \{0,1\}^{\bb Z}$. For $x \in \bb Z$ let $\tau_x: \Omega \to \Omega$ denote the canonical shift: $\tau_x \eta(y) = \eta(x+y)$ for any $\eta \in \Omega$ and any $y \in \bb Z$.  We say that the {\em support} of $f$ is {\em contained} in a set $A \subseteq \bb Z$ if $f(\eta) = f(\xi)$ whenever $\eta(x) =\xi(x)$ for every $x \in A$. We say that $f$ is a \emph{local function} if its support is contained in some finite set. Let $c: \Omega \to [0,\infty)$ satisfy
\begin{itemize}
\item[i)] {\em Finite range:} $c(\cdot)$ is a local function;
\item[ii)] {\em Ellipticity:} There exists $\epsilon_0>0$ such that $c(\eta) \geq \epsilon_0$ for any $\eta \in \Omega$;
\item[iii)] {\em Reversibility: } $c(\eta) =c(\xi)$ whenever $\eta(x) = \xi(x)$ for all $x \neq 0,1$, that is, the support of $c(\cdot)$ is contained in $\bb Z \setminus \{0,1\}$.
\end{itemize}

Let $c_x:\Omega \to \bb R$ be defined as $c_x(\eta) = c(\tau_x \eta)$ for any $\eta \in \Omega$. For $f: \Omega \to \bb R$ local, let $L_b f: \Omega  \to \Omega$ be defined as 
\[
L_b f(\eta) = \sum_{x \in \bb Z} c_x(\eta) \lrc{f(\eta^{x,x+1}-f(\eta))}
\]
where $\eta^{x,x+1}$ is defined as
\[
\eta^{x,x+1}(z) =
\left\{
\begin{array}{c@{\;;\;}l}
\eta(x+1) & z=x,\\
\eta(x) & z=x+1,\\
\eta(z) & z \neq x,x+1.
\end{array}
\right.
\]
Since $f$ is local, only a finite number of terms in the sum defining $L_b f$ are non-zero.

\textbf{The {\em lattice gas} with interaction rate $c(\cdot)$ is the Markov process $(\eta_t)_{t \geq 0}$ defined in $\Omega$ and generated by the operator $L_b$.} Notice that the SSEP corresponds to the choice $c \equiv 1$. 

For $\rho \in [0,1]$, let $\nu_\rho$ be the Bernoulli product measure in $\Omega$: for any $x_1,\dots,x_\ell \in \bb Z$,
\[
\nu_\rho\lrch{\eta(x_1)=\cdots=\eta(x_\ell)=1} = \rho^\ell.
\]
 Thanks to the reversibility condition iii), these measures are invariant under the evolution of $(\eta_t)_{t \geq 0}$. From now on, we fix $\rho \in (0,1)$ and we assume that $\eta_0$ (and therefore $\eta_t$ for any $t \geq 0$) has law $\nu_\rho$.

\bigskip 

 Let $\mc R \subseteq \bb Z \setminus \{0\}$ be a finite set. For each $z \in \mc R$, let $r_z: \Omega \to [0,\infty)$ be a local function. Let $n \in \bb N$ be a scaling parameter and let $(\eta_t^n)_{t \geq 0}$ be the lattice gas defined above, speeded up by $n^2$, that is, $\eta_t^n = \eta_{tn^2}$. We denote by $\bb P_n$ the law of $(\eta_t^n)_{t \geq 0}$ and we denote by $\bb E_n$ the expectation with respect to $\bb P_n$. For $x \in \bb Z$ and $z \in \mc R$, define $r_z(\cdot,x): \Omega \to [0,\infty)$ as $r_z(\eta, x) = r_z(\tau_x \eta)$ for any $\eta \in \Omega$. 

\textbf{We define the process $(x_t^n)_{t \geq 0}$ as the random walk that jumps from $x$ to $x+z$ with instantaneous rate $n\, r_z(\eta_t^n, x)$.} The pair $\{(\eta_t^n, x_t^n); t \geq 0\}$ turns out to be a Markov process, generated by the operator
\[
\begin{split}
L_n f(\eta,x) 
		&= n^2 \sum_{y \in \bb Z} c_y(\eta) \lrc{ f(\eta^{y,y+1},x)-f(\eta,x)} \\
		&\quad \quad + n \sum_{z \in \mc R} r_z(\eta,x) \lrc{ f(\eta,x+z)-f(\eta,x)}.
\end{split}
\]
Define 
\begin{equation}\label{asymptotic_speed}
v(\rho) = \int \sum_{z \in \mc R} z r_z \,d \nu_\rho.
\end{equation}
Let us denote by $\mc D([0,T], \bb R)$ the space of c\`adl\`ag, real-valued trajectories.
The following result was proved in \cite{afjv}:

\begin{proposition}\label{lln}
	 Let $ v(\rho) $ be as in \eqref{asymptotic_speed}. Then, for any $T >0$, 
\label{p1}
\[
\lim_{n \to \infty} \frac{x_t^n}{n} = v(\rho)\cdot t
\]
in law with respect to the $J_1$-Skorohod topology of $\mc D([0,T],\bb R)$.
\end{proposition}

Notice that Proposition \ref{p1} can be interpreted as a law of large numbers for the random walk.
In this article we will prove the corresponding central limit theorem:

\begin{theorem}
\label{t1}
For any $T >0$,
\[
\lim_{n \to \infty} \frac{x_t^n - v(\rho) t n}{\sqrt n} = \sigma B_t + Z_t,
\]
in law with respect to the $J_1$-Skorohod topology of $\mc D([0,T],\bb R)$. In the above display,
\begin{equation}\label{sigma}
\sigma^2 =\int  \sum_{z \in \mc R} z^2 r_z d \nu_\rho
\end{equation}
and $(Z_t)_{t \geq 0}$ is a Gaussian process of stationary increments, independent from $(B_t)_{t \geq 0}$. 
\end{theorem}

\subsection{The environment process}

A classical idea in the context of random walks in random environments is to consider the {\em environment as seen by the random walk}. Here we follow the approach of \cite{kipnis1986central}. The process $(\xi_t^n)_{t \geq 0}$ with values in $\Omega$, defined as
\[
\xi_t^n(x) = \eta_t^n(x + x_t^n) \text{ for any } x \in \bb Z \text{ and any } t \geq 0
\]
is a Markov process generated by the operator $L_n = n^2 L_b+  nL^{rw}$, where
\[
L^{rw} f( \xi) = \sum_{z \in \mc R} r_z(\xi) \big( f(\tau_z \xi)-f(\xi)\big).
\]

The process $(x_t^n)_{t \geq 0}$ can be recovered from $(\xi_t^n)_{t \geq 0}$ as follows: for each $z \in \mc R$, let $N_t^{z,n}$ be the number of shifts in direction $z$ the process $(\xi_t^n)_{t \geq 0}$ has performed up to time $t$. On one hand,
\[
x_t^n = \sum_{z \in \mc R} z N_t^{z,n},
\]
and on the other hand, $(N_t^{z,n})_{t \geq 0}$ is a (time-inhomogeneous) Poisson process of rate $n \,r_z(\xi_t^n)$. Therefore,
\[
M_t^{z,n} := \frac{1}{\sqrt n} N_t^{z,n} - \sqrt n \int_0^t r_z(\xi_s^n) ds
\]
is a martingale with respect to the filtration $ \mc F_t =\sigma \{\xi^n_s:s\leq t\}$. Its predictable quadratic variation is given by
\[
\lrb{M_t^{z,n}} = \int_0^t r_z(\xi_s^n)ds.
\]
Moreover, since the jumps of these Poisson processes are disjoint, these martingales are mutually orthogonal.
\bigskip 

Adding the martingales $(M^{n,z}_t)_{t\geq 0}$, we can write the position of the random walk as a sum of a martingale and an integral term, namely
\begin{equation}\label{decomposition}
\frac{x_t^n - v(\rho) n t}{\sqrt n} = M_t^n + A_t^n ,
\end{equation}
where $M^n_t:=\sum_{z\in\mc R}M^{n,z}_t$ and 
\begin{equation}\label{gato}
A_t^n = \sqrt n \int_0^t \big(\omega(\xi_s^n)-v(\rho)\big) ds, \mbox{ with }
\end{equation}
$\omega(\xi) := \sum_{z \in \mc R} z\, r_z(\xi)$. Besides, 

\[
\lrb{M_t^n} = \int_0^t \sum_{z \in \mc R} z^2 r_z(\xi_s^n) ds.
\]
The process $(A_t^n)_{t \geq 0}$ is an instance of what is known in the literature as an {\em additive functional} of the chain $(\xi_t^n)_{t \geq 0}$.
Theorem \ref{t1} is an immediate consequence of the following result:

\begin{theorem}
\label{t2}
Consider the decomposition \eqref{decomposition} and recall the definition of $\sigma$ in \eqref{sigma}. Fix $T >0$.
\begin{itemize}
\item[i)] As $n \to \infty$, $\{M_t^n; t \in [0,T]\}_{n \in \bb N}$ converges in law to $(\sigma B_t)_{t \in [0,T]}$ with respect to the $J_1$-Skorohod topology of $\mc D([0,T],\bb R)$;
\item[ii)] as $n \to \infty$, the sequence $\{A_t^n; t \in [0,T]\}_{n \in \bb N}$ converges in law to $\{Z_t :t \in [0,T]\}$ with respect to the $J_1$-Skorohod topology of $\mc D([0,T],\bb R)$;
\item[iii)] The processes $\{B_t :t \in [0,T]\}$ and $\{Z_t :t \in [0,T]\}$ are independent.
\end{itemize}
\end{theorem}

\subsection{Auxiliary results}

As the reader can guess from the statement of Theorem \ref{t2}, we will need distinct tools to tackle the convergence of the martingale part of $x_t^n$ and the additive functional part of $x_t^n$. Fortunately, the machinery for the martingale part can be found in the literature.

\subsubsection{Invariance principle for martingales}

In order to prove convergence of the processes $\{M_t^n; t \geq 0\}_{n \in \bb N}$, as well as the independence of the limiting objects $(B_t)_{t \geq 0}$, $(Z_t)_{t \geq 0}$, we will use the following result:

\begin{proposition}[Martingale CLT] \label{p1}
Let $\{\mc M_t^n, t \in [0,T] \}_{n \in \bb N}$ be a sequence of square-integrable martingales. Assume that:
\begin{itemize}
\item[i)] the sequence of predictable quadratic variation processes $\{\lrb{\mc M_t^n}; t \in [0,T]\}$ converges in law to an increasing function $H:[0,T]\to\bb R$;

\item[ii)] the size of the largest jump of $(\mc M_t^n)_{t \geq 0}$ converges in probability to $0$.
\end{itemize}

Then $\{\mc M_t^n; t \geq 0\}_{n \in \bb N}$ converges in law to a continuous martingale of quadratic variation $H$. 
In addition, let $\{\mc N_t^n; t \geq 0\}_{n \in \bb N}$ be another sequence of square-integrable martingales satisfying i), ii). If $(\mc M^n_t)_{t \geq 0}$ is orthogonal to $(\mc N_t^n)_{t \geq 0}$ for each $n$, then the limiting martingales are independent.
\end{proposition}

A proof of this result for the case $H(t) = \sigma t$ can be found in \cite{whitt2007proofs}. The proof for general $H(t)$ can be found in Chapter VIII.3.a of \cite{jacod2013limit}.

\subsubsection{Density fluctuation field}

The process $(A_t^n)_{t \geq 0}$ defined as
\[
A_t^n = \sqrt{n} \int_0^t \big( \omega(\xi_s^n) - v(\rho)\big) ds
\]
is a particular instance of what is called an {\em additive functional} of the Markov process $(\xi_t^n)_{t \geq 0}$. In order to prove the convergence of these additive functionals, we will follow the strategy introduced in \cite{slafips}, that relates additive functionals with the {\em density fluctuation field} of the underlying particle system. In order to describe this field, we need some definitions. 

Let $\mc S(\bb R)$ be the Schwarz space of test functions in $\bb R$. For $f \in \mc S(\bb R)$, $n \in \bb N$ and $t \geq 0$, let 
\begin{equation}\label{fluctuation_field}
X_t^n(f) := \frac{1}{\sqrt n} \sum_{x \in \bb Z} \big(\eta_t^n(x) -\rho\big) f\big(\tfrac{x}{n}\big).
\end{equation}
By duality, this relation defines a process $(X_t^n)_{t \geq 0}$ with values in the space $\mc S'(\bb R)$ of tempered distributions. This is the {\em density fluctuation field} associated to the process $(\eta_t^n)_{t \geq 0}$. 

Since the topology of the space of distributions is not very strong, it is sometimes more convenient to consider the {\em Sobolev spaces} $H_{\ell}(\bb R)$ instead, defined as the closure of $\mc S(\bb R)$ with respect to the norms
\[
\|f\|_{H_{\ell}(\bb R)} = \Big(\int f(x)(-\Delta +x^2)^\ell f(x) dx\Big)^{1/2}.
\] 
One can check that $(X_t^n)_{t \geq 0}$ is a well-defined process in $H_{-2}(\bb R)$. The following result was proved in \cite{chang1994equilibrium}:

\begin{proposition}    \label{p2}
Let $(\eta_t^n)_{t \geq 0}$ the lattice gas with initial law $\nu_\rho$. There exists a constant $D(\rho)$ such that for any $T >0$,
\[
\lim_{n \to \infty} X_t^n =X_t
\]
in law with respect to the $J_1$-Skorohod topology of $\mc D([0,T],H_{-2}(\bb R))$, where $(X_t)_{t \geq 0}$ is the stationary solution of
\begin{equation}\label{ou}
\partial_t X = D(\rho) \Delta X + \sqrt{2D(\rho)\rho(1-\rho)} \nabla \dot{\mc W}_t.
\end{equation}
In this equation, $\dot{\mc W}_t$ denotes a standard, space-time white noise.
\end{proposition}

\section{Replacement lemma and entropy bound}

In this section we will establish two estimates that are fundamental to the proof of Theorem \ref{t2}. First, we obtain a sharp bound on the entropy production for the environment process. Then, we prove the so-called {\em replacement lemma}, that allows to write $A_t$ as a function of the density of particles plus an error that vanishes in the limit. 

\subsection{Entropy bound}\label{entropy_section}

Let us recall that the processes $(\eta_t^n)_{t \geq 0}$ and $(\xi_t^n)_{t \geq 0}$ start from the Bernoulli product law $\nu_\rho$. We recall that $\nu_\rho$ is invariant under the evolution of $(\eta_t^n)_{t \geq 0}$ and stress that it is not invariant under $(\xi_t^n)_{t \geq 0}$. Let $\mu_t^n$ be the law of $\xi_t^n$ and define
\[
H_n(t) := H(\mu_t^n | \nu_\rho), 
\]
where 
\[
H(\mu|\nu) := \int f \log f d\nu, \quad f = \frac{d\mu}{d\nu}
\]
is the relative entropy (or Kullback-Leibler divergence) of $\mu$ with respect to $\nu$. The main result of this section is the following bound:

\begin{theorem}\label{entropy_bound}
\label{t3} There exists $C$ depending only on $\rho$, $\{r_z; z \in \mc R\}$ and $\epsilon_0$ such that $H_n'(t) \leq C$ for any $t \geq 0$. In particular, $H_n(t) \leq Ct$ for any $t \geq 0$.
\end{theorem}

\begin{remark}
In \cite{afjv} it is proved that $H'_n(t) \leq Cn$. As observed in \cite{bertini2003large}, a bound of this type is enough (aside from the usual model-dependent technical points) to adapt Varadhan's approach to obtain hydrodynamic limits and the associated large deviations principle. In \cite{afjv}, \cite{avena2009large}, this strategy was successfully applied for the process $(\xi_t^n)_{t \geq 0}$. Actually, the bound $H_n'(t) \leq Cn$ is not hard to prove (see Lemma 2.2 in \cite{afjv}, Lemma 3.2 in \cite{bertini2003large} or Lemma 6.1 in \cite{farfan2011hydrostatics}). A bound of the form
\[
\lim_{n \to \infty} \frac{H_n(t)}{n} =0
\]
is more difficult to obtain, and is the main point of the so-called Yau's {\em relative entropy} method in hydrodynamic limits, see \cite{yau1991relative} and Chapter 6 of \cite{kl}. Surprisingly, an adaptation of Yau's method to the model considered in this article only gives a bound of the form
\[
\lim_{n \to \infty} \sup_{0 \leq s \leq t} \frac{H_n'(t)}{n} =0,
\]
which is very far from Theorem \ref{t3}.
\end{remark}

\begin{proof}
Let $f_t$ be the Radon-Nykodim derivative of $\mu_t^n$ with respect to $\nu_\rho$ (we are not indexing in $n$ in order not to overcharge the notation). By Theorem A.9.2 in \cite{kl}, we have that
\[
H_n'(t) \leq 2 \lrb{\sqrt{f_t}, L_n \sqrt{f_t}}.
\]
In the last equation and throughout the rest of the article, $ \lrb{\cdot,\cdot} $ denotes the inner product in $ L^2(\nu_{\rho}) $. 

Recall that $L_n = n^2 L_b+nL^{rw}$. Since $\nu_\rho$ is invariant under $L_b$, we have that $\lrb{\sqrt{f_t}, L_b \sqrt{f_t}} \leq 0$. Even more, $\nu_{\rho}$ is reversible for $L_b$. From reversibility one can show
\begin{equation}\label{wartortle}
\lrb{\sqrt{f_t}, L_b \sqrt{f_t}} = -\frac{1}{2}\sum_{x \in \bb Z} \int c_x(\xi) ( \sqrt{f_t(\xi^{x,x+1})}-\sqrt{f_t(\xi)})^2 d\nu_\rho.
\end{equation}
Let us introduce the {\em Dirichlet form} $\mc D(\cdot)$, defined as
\begin{equation}\label{dirichlet_form_definition}
\mc D(h) = \frac{1}{2}\sum_{x \in \bb Z} \int ( h(\xi^{x,x+1})-h(\xi))^2 d\nu_\rho
\end{equation}
for any $h: \Omega \to \bb R$. Thanks to the ellipticity condition $c_x \geq \epsilon_0$ and to \eqref{wartortle}, we see that
\begin{equation}
\label{pavo}
H_n'(t) \leq - \epsilon_0\, n^2\,\mc D(\sqrt{f_t})  + 2n \lrb{\sqrt{f_t}, L^{rw} \sqrt{f_t}}.
\end{equation}
Therefore, if we are able to control $\lrb{\sqrt{f_t}, L^{rw} \sqrt{f_t}}$ in terms of the Dirichlet form of $\sqrt{f_t}$, the theorem will be proved. The following lemma provides the required bound, which is going to be used several times in the remaining of the article.
\end{proof}

\begin{lemma}\label{lema_fundamental}
For any $f\geq 0$ such that $\int f\,d\nu_{\rho}=1$, the following inequality holds:

\begin{equation}\label{squirtle}
\lrb{\sqrt f, L_n \sqrt f} \leq -\epsilon_0 \,n^2 \mc D(\sqrt f) + \lrb{\psi, n f},
\end{equation}
where 
\begin{equation}\label{charmander}
\psi(\xi):= \frac{1}{2}\sum_{z \in \mc R} (r_z(\tau_{-z} \xi) -r_z(\xi)).
\end{equation}
In addition, for all $\beta >0$,
\begin{equation}\label{blastoise}
\lrb{n\psi, f}\leq \beta \mc D(\sqrt{f}) + \frac{Cn^2}{\beta},
\end{equation}
where $C>0$ does not depend on $n$.
\end{lemma}

\begin{proof}

Combining \eqref{wartortle} and the ellipticity assumption, we get 

\begin{equation}
2\lrb{\sqrt f, L_n \sqrt f} \leq -\epsilon_0 n^2 \mc D(\sqrt f) + 2n \lrb{\sqrt f, L^{wr} \sqrt f}.
\end{equation}

Using the identity $\sqrt a(\sqrt b - \sqrt a) = -\frac{1}{2}(\sqrt b -\sqrt a)^2 + \frac{1}{2}(b-a)$, we get 
\[
\begin{split}
\lrb{\sqrt{f}, L^{rw} \sqrt{f}}
		&= \sum_{z \in \mc R} \int r_z(\xi) \sqrt{f(\xi)} ( \sqrt{f(\tau_z\xi)}-\sqrt{f(\xi)})d\nu_\rho\\
		&=-\frac{1}{2} \sum_{z \in \mc R} \int r_z(\xi) ( \sqrt{f(\tau_z\xi)}-\sqrt{f(\xi)})^2d\nu_\rho\\
		&\quad \quad +  \frac{1}{2}\sum_{z \in \mc R} \int r_z(\xi) ( f(\tau_z\xi)-f(\xi))d\nu_\rho.
\end{split}
\]
Neglecting the first term and performing the change of variables $\xi \mapsto \tau_z \xi$, we conclude 
\begin{equation}\label{adjoint}
\lrb{\sqrt{f}, L^{rw} \sqrt{f}} \leq \frac{1}{2}\sum_{z \in \mc R} \lrb{f, \,r_z \circ \tau_{-z} - r_z},
\end{equation}
and this finishes the proof of \eqref{squirtle}.

It remains to prove \eqref{blastoise}. Let us start with the function $r\circ \tau_1  - r$, where $r$ is local. We are going to write $r\circ \tau_1 - r$ as a sum of terms of the form $h^{x,x+1}-h$ and apply Lemma \ref{l1}.

To simplify the notation, assume that $r$ has support in $\{0,\ldots, k\}$ and denote $\nabla^{x,x+1}h:=h^{x,x+1}-h$. Then $r(\tau_1\xi)=r(\nabla^{k,k+1}\cdots \nabla^{0,1}\xi)$. Therefore

\begin{equation}
r(\tau_1\xi) - r(\xi) = r(\nabla^{0,1}\xi) - r(\xi) + \sum_{y=1}^k  r(\nabla^{y,y+1}\cdots \nabla^{0,1}\xi) - r(\nabla^{y-1,y}\cdots \nabla^{0,1}\xi).
\end{equation}

Applying Lemma \ref{l1}, we get, for any $\beta >0$,

\begin{equation}
\lrb{nf, r\circ \tau_1  - r} \leq \beta \mc D(\sqrt f) + \frac{n^2}{\beta}||r||_{\infty}^2 k.
\end{equation}

Using a telescoping argument, we can obtain a similar bound for \\ $\lrb{nf,r\circ \tau_z - r}$. Summing over $z\in \mc R$ we finish the proof. Notice that the constant $C$ in the statement depends on $\psi$. It is a function of the size of $\mc R$, the sizes of the supports of the $r_z$ and the numbers $||r_z||_{\infty}$.
%
\end{proof}

\subsection{Replacement lemma}

Let $ \varphi:\RR\to\RR_+ $ be a smooth function with compact support in $ (0,1) $ and such that $ \int_{\RR}\varphi(u)\,du=1 $. Let $ \varphi_{\epsilon}(u):=\epsilon^{-1}\varphi(u/\epsilon) $.

In this section we will prove that the additive functional $A_t^n$ is asymptotically equivalent to a function of the density of particles around the origin. More precisely, we will prove that
\[
\limsup_{\epsilon \to 0} \limsup_{n \to \infty} \bb P_n( | \sqrt n \int_0^t  \omega(\xi_s^n) - v(\rho) - v'(\rho) (\xi^n_s \star \varphi_\epsilon)(0) \,ds | > \delta )  =0,
\]
where we use the notation
\begin{equation}\label{def_convolution}
(\varphi_\epsilon\star \xi)(x):=
\frac{1}{n}\sum_{y\in\ZZ} \varphi_\epsilon\lrp{\frac{y}{n}}(\xi(x+y)-\rho).
\end{equation}
In this theorem the particular form of the function $\omega$ does not play a fundamental role. In fact, this result is a particular instance of what is known in the literature as the {\em replacement lemma}, which roughly states that any local function of $\xi_t^n$ is asymptotically equivalent to a function of the density of particles around the origin. We take averages using a smooth function instead of the usual arithmetic mean for technical reasons having to do with the topology of Skorohod space. This issue  shows up in Section \ref{sec_characterization}, where we characterize the limiting trajectories of the random walk. 

\begin{theorem}[Replacement lemma]
\label{t4}
Let $\phi: \Omega \to \bb R$ be a local function. For $\lambda \in [0,1]$, define $\bar{\phi}(\lambda) = \int \phi \,d \nu_\lambda$.
Then, for any $\delta > 0$ and any $t \geq 0$,
\[
\limsup_{\epsilon \to 0} \limsup_{n \to \infty} \bb P_n( | \sqrt n\int_0^t  \phi(\xi_s^n) - \bar \phi(\rho) - \bar \phi'(\rho) (\xi^n_s \star \varphi_\epsilon)(0)  \,ds | > \delta )  =0.
\]
\end{theorem}

\begin{proof}
First we observe that for any random variable $X$,
\begin{equation}
 P(|X|>\delta) \leq P(X>\delta) + P(-X >\delta).
\end{equation} 
Considering $\phi$ and $-\phi$, it is enough to prove that
\begin{equation}\label{pato}
\limsup_{\epsilon \to 0} \limsup_{n \to \infty} \bb P_n( \sqrt n \int_0^t  \phi(\xi_s^n) - \bar \phi(\rho) - \bar \phi'(\rho) (\xi^n_s \star \varphi_\epsilon)(0) \,ds > \delta )  =0.
\end{equation}
Before entering into the details of the proof, let us see how are we going to take advantage of Lemma \ref{l2}. For any bounded function $V$ and any positive $\gamma$ and $ \beta $, 
\begin{equation}\label{pichu}
\begin{split}
\log \bb P_n( \int_0^t V(\xi_s^n) ds > \delta ) 
		&\leq -\gamma \delta  + \log \bb E_n[ e^{\gamma \int_0^t V(\xi_s^n) ds }]\\
		&\leq -\gamma \delta + t\cdot\sup_{f}
		\lrch{
\lrb{\gamma V, f} + (\beta - \epsilon_0)n^2\mc D(\sqrt f) + \frac{C}{\beta}	
	},
\end{split}
\end{equation}
where the supremum is taken over all $ f\geq 0 $ such that $ \int f\,d\nu_{\rho}=1 $ and where $ C >0$ does not depend on $ n $. The expression above becomes easier to remember if one keeps in mind that the term $ -\epsilon_0 \mc D(\sqrt f) $ comes from the reversible dynamics and the term with $ \beta $ comes from the random walk dynamics.

Let $ \mc R\subset \ZZ $ be the support of $ \phi $. The first thing to notice is that every mean-zero local function $ \phi $ can be written as a linear combination of the simpler variables $\{ \xi(A): A\subset \mc R\} $, where
\begin{equation}\label{key}
\xi(A):=\prod_{x\in A}(\xi(x)-\rho).
\end{equation}

 It is enough, then, to prove inequality \eqref{pato} when $ \phi(\xi)$ is of the form $ \phi(\xi)=\xi(A) $ for some finite set $ A $.

We start with the simplest case, which is 

\begin{equation}\label{replacement_one_site}
\varlimsup_{\epsilon \to 0}\varlimsup_{n \to \infty}\PP \lrp{
\Big|
	\int_0^t \xi^n_s(0) - \rho - (\xi^n_s\star \varphi_{\epsilon})(0)\,ds
\Big| > \delta
} = 0.
\end{equation}

Since time will not play any role in the computations that follow, we will omit it from the notation for a while. Denote $ \xi_x:=\xi(x) $ and  $ \xxi_0:=	\xi_0-\rho $.  Recall that $ \lrb{\cdot, \cdot} $ denotes the inner product in $ L^2(\nu_{\rho}) $ and that $ \mc D(\sqrt f) $ denotes the Dirichlet form of symmetric exclusion, as defined in \eqref{dirichlet_form_definition}.

Back to the proof of \eqref{replacement_one_site}. In view of \eqref{pichu}, we need to estimate the integral $\lrb{\xxi_0-(\xi\star \varphi_\epsilon)(0),f} $ in terms of $ \mc D(\sqrt f) $. More precisely, we are going to prove that, for any $ \nu_{\rho} $-density $ f $, the following inequality holds:
\begin{equation}\label{flow_ineq}
\gamma  n^{\frac{1}{2}}\lrb{ \xxi_0-(\xi\star\varphi_\epsilon)(0),f} \leq \alpha n^2 \mc D(\sqrt f) + \epsilon\gamma^2\frac{C'}{\alpha} + o_n(1),
\end{equation}
where $ \alpha > 0 $ is arbitrary and $ C' $ does not depend on $ n $.

\bigskip

We would like to write the difference inside the inner product in \eqref{flow_ineq} as a telescoping sum, in order to apply Lemma \ref{l1}. For that we need the coefficients of the $ \xi_x -\rho$ to sum up to $ 1 $, what they almost do. Define
\begin{equation}\label{key}
m_n:=\frac{1}{n}\sum_{x\in \ZZ} \varphi\lrp{\frac{x}{n}}
\end{equation}
and write the telescoping sum
\begin{equation}\label{telescoping_sum}
\begin{aligned}
& \gamma n^{\frac{1}{2}}\lrb{\xxi_0 - \frac{1}{n}\sum_{x\in \ZZ}\varphi_\epsilon\lrp{\frac{x}{n}}\xxi_x, f}\\
=\, &
\gamma n^{\frac{1}{2}}(1-m_n)\lrb{\xxi_0,f}+ 
 \sum_{x\in \ZZ}
\sum_{y=x+1}^{\infty}\varphi\lrp{\frac{y}{n}}
\gamma n^{-\frac{1}{2}}\lrb{\xxi_x-\xxi_{x+1},f
}  .
\end{aligned}
\end{equation}

Since $ m_n $ is a Riemman sum for $ \int_0^\epsilon \varphi_\epsilon(u)\,du =1$, the first term is of order $ n^{-\frac{1}{2}} $. As for the second term, notice that, since $ \varphi_\epsilon $ has support contained in $ (0,\epsilon) $, only finitely many terms of the sum over $ x $ are not null, namely those with $ 0< x < \epsilon n $. 

Fix $ \alpha >0 $. Applying Lemma \ref{l1}, we can bound the second term by

\begin{equation}\label{replacement_one_site_var}
\alpha n^2 \mc D(\sqrt f) + \frac{2\gamma^2}{\alpha n} \sum_{x=0}^{\epsilon n}
\lrp{\frac{1}{n}\sum_{y=x+1}^{\epsilon n}\varphi_{\epsilon}\lrp{\frac{y}{n}}}^2.
\end{equation}
Using that  $ 0\leq \varphi_{\epsilon} \leq \epsilon^{-1}\norm{\varphi}_{\infty} $,  we get the inequality

\begin{equation}\label{l2_tail}
\sum_{x=0}^{\epsilon n}
\lrp{\frac{1}{n}\sum_{y=x+1}^{\epsilon n}\varphi_{\epsilon}\lrp{\frac{y}{n}}}^2\leq \norm{\varphi}^2_{\infty}\epsilon n.
\end{equation}

Therefore, expression \eqref{replacement_one_site_var} is bounded from above by

\[
\alpha n^2 \mc D(\sqrt f) + \frac{2\norm{\varphi}_{\infty} \gamma^2 \epsilon }{\alpha} ,
\]
and this finishes the proof of \eqref{flow_ineq}. Plugging this inequality into \eqref{pichu}, with the choices $ \alpha = \beta = \frac{\epsilon_0}{2} $, we get the Replacement Lemma when the local function is $ \phi(\xi)=\xi_0-\rho $. In an analogous manner, one can prove the lemma for $ \phi(\xi)=\xi_x-\rho $ for any $x\in \mc R$. 
 
Next, we show that the higher order monomials vanish. More precisely, we show that if $ A\subset \ZZ $ is a finite set and $ |A|\geq 2 $ then 

\begin{equation}\label{higher_degree_vanishes}
\limsup_{\epsilon \to 0}\limsup_{n \to \infty}\PP
\lrp{\Big|
\sqrt n\int_0^t  \xi^n_s(A)\,ds \Big| > 
\delta 
}
=0.
\end{equation}

Write the set $ A $ in the form $ A = \{x_0\}\cup A' \cup \{y_0\} $, where we assume that $ x_0<y_0 $ and $ A' \subset \{x_0+1,\ldots, y_0-1\} $. Denote by $ (\xi \star \tilde{\varphi_\epsilon})(x_0) $ the weighted average of the centered configuration $ \xi $ in a box to the left of $ x_0 $:

\[
(\xi \star \tilde{\varphi_\epsilon})(x_0) = \frac{1}{n}\sum_{y\in\ZZ}(\xi(x_0-y)-\rho)\varphi_{\epsilon}\lrp{\frac{y}{n}} = \frac{1}{n}\sum_{y\in\ZZ}\varphi_{\epsilon}\lrp{\frac{y}{n}}	\xxi_{x_0-y}.
\]


To prove assertion \eqref{higher_degree_vanishes}, we prove that each of the probabilities below converges to zero as first $ n\to \infty $ then $ \epsilon \to 0 $.
\begin{equation}
\begin{aligned}
\PP
\lrp{\Big|
	\sqrt n \int_0^t   \xi^n_s(x_0)
	\xi^n_s(A') \lrp{
		\bar{\xi}^n_s(y_0) - (\xi^n_s\star \varphi_{\epsilon})(y_0)
	}
	\,ds 
	\Big|
	> \delta 
}
& \\
\PP
\lrp{\Big|
	\sqrt n \int_0^t   (\bar{\xi^n_s}(x_0) - (\xi^n_s\star \tilde{\varphi_{\epsilon}})(x_0))
	\xi^n_s(A') \lrp{
		\bar{\xi}^n_s(y_0) - (\xi^n_s\star \varphi_{\epsilon})(y_0)
	}
	\,ds 
	\Big|
	> \delta 
}
& \\
\PP
\lrp{\Big|
	\sqrt n \int_0^t
	(\xi^n_s\star \tilde{\varphi_{\epsilon}})(x_0)
	\xi^n_s(A')
	(\xi^n_s\star \varphi_{\epsilon})(y_0)
	\,ds 
	\Big|
	> \delta 
}
\end{aligned}
\end{equation}

To bound the first probability, we mimic the proof of \eqref{replacement_one_site}.  That is, first we use \eqref{pichu} to reduce the proof to a variational problem, then we write the difference $ \bar{\xi}^n_s(y_0) - (\xi^n_s\star \varphi_{\epsilon})(y_0)$ as a telescoping sum as in \eqref{telescoping_sum} and apply Lemma \ref{l1} to each term of the sum, with the roles of $ g $ and $ h $ in that lemma being played by $ \bar{\xi}_x $ and $ \bar{\xi}(x_0)\xi(A') $ respectively. The proof can be replicated to bound the second inequality, this time with the roles of $ g $ and $ h $ being played by $ \xxi_x $ and $ \xi^n_s(A') \lrp{
	\bar{\xi}^n_s(y_0) - (\xi^n_s\star \varphi_{\epsilon})(y_0)
} $, respectively. In both cases, one can use the bounds $ \norm{h}_{\infty}\leq 1 $ and \eqref{l2_tail}.

It remains to deal with the last probability. For that, recall that $ \mu^n_s $ denotes the law of $ \xi^n_s $, the environment as seen from the random walk at time $ s $. We claim that there exists a large $ D = D(t) $ such that, for all $ s\leq t $,

\begin{equation}
\lim_{n\to\infty}	\EE_{n}\lrc{
	\Big|
	\sqrt n 
	(\xi^n_s\star \tilde{\varphi_{\epsilon}})(x_0)
	\xi^n_s(A')
	(\xi^n_s\star \varphi_{\epsilon})(y_0)
	\Big|
}
=0.
\end{equation}

To prove that, we start with Cauchy-Schwarz inequality and the bound $ |\xi(A')| \leq  1$.

\begin{equation}\label{key}
\begin{aligned}
\EE_{\mu^n_s}\lrc{
		\Big|
	\sqrt n 
	(\xi\star \tilde{\varphi_{\epsilon}})(x_0)
	\xi(A')
	(\xi\star \varphi_{\epsilon})(y_0)
	\Big|
}
& \leq 
\frac{1}{2}\,\EE_{\mu^n_s}\lrc{
		\sqrt n
	(\xi\star \tilde{\varphi_{\epsilon}})^2(x_0)
}\\
& +
\frac{1}{2}\,\EE_{\mu^n_s}\lrc{
	\sqrt n
	(\xi\star {\varphi_{\epsilon}})^2(y_0)
}.
\end{aligned}
\end{equation}
We can use relative entropy to replace the expectation under $ \mu^n_s $ by an expectation under $ \nu_{\rho} $ using the following argument: for any $ \alpha >0 $, it holds
\begin{equation}\label{replacement_degree_2}
\EE_{\mu^n_s}\lrc{
	\sqrt n
	\lrp{
		\xi\star \varphi_{\epsilon}
	}^2	(y_0)
}
\leq 
\frac{H(\mu^n_s|\nu_\rho)}{\alpha} 
+\frac{1}{\alpha}
\log 
\EE_{\nu_\rho}\lrc{
	e^{
	\alpha \sqrt n
	\lrp{
		\xi\star \varphi_{\epsilon}
		}^2	(y_0)
	}
}.
\end{equation}

Under the product measure $ \nu_{\rho} $, the random variable in the exponent is a linear combination of i.i.d. random variables, recall  notation \eqref{def_convolution}. The variance of this sum is at most $ \frac{\norm{\varphi}_{\infty}}{\epsilon n} $.   By Lemma \ref{subgaussianity}, the logarithm is bounded by $	\frac{2\alpha \norm{\varphi}_{\infty}}{\epsilon \sqrt n}  $ whenever this quantity is smaller than $ \frac{1}{2} $, that is, whenever $ \alpha \leq \frac{ \epsilon \sqrt n}{4 \norm{\varphi}_\infty }$. Going back to \eqref{replacement_degree_2}, we can choose such an $ \alpha $ of order $ \sqrt n $ and the fact that the entropy is of order $ 1 $, Theorem \ref{entropy_bound}, to conclude

\begin{equation}\label{key}
\lim_{n\to\infty}\EE_{\mu^n_s}\lrc{
	 \sqrt n
	\lrp{
		\xi\star \varphi_{\epsilon}
	}^2	(y_0)
}
=0.
\end{equation}

A parallel argument shows
\begin{equation}\label{key}
\lim_{n\to\infty}\EE_{\mu^n_s}\lrc{
	\sqrt n
	\lrp{
	\xi(A') \star \tilde{\varphi_\epsilon}
}^2	(x_0)
}
=0,
\end{equation}
and this finishes the proof of \eqref{higher_degree_vanishes}.

\end{proof}

\section{Tightness}

In this section we prove that the sequence of additive functionals 
\[
\{A^n_t:t\in [0,T]\}_{n\in\mb N},
\] defined in \eqref{gato}, is tight in $C_{[0,T]}\mb R$. Since $A^n_0 =0$ for all $n\in\mb N$, we only need to prove equicontinuity. 

\bigskip

  The proof is an application of the Kolmogorov-Centov criterion, see Problem 2.4.11 in \cite{karatzas}.

\begin{proposition}\label{kolmogorov_centov}
	Assume that the sequence of stochastic processes $ \{X^n_t:
        t\in [0,T]\} _{n\in\mb N}$  satisfies 
	\begin{equation}
	\varlimsup_{n\to\infty}\mb E[|X^n_t - X^n_s|^\lambda]\leq C|t-s|^{1+\lambda '}
	\end{equation}
	for some positive constants $ \lambda $, $ \lambda ' $ and $ C $ and for all $ s,t\in [0,T] $. Then it also satisfies
	\begin{equation}\label{modulusofcontinuity}
	\lim_{\delta\to 0}\varlimsup_{n\to\infty}\mb P\left(
	\sup_{\substack{|t-s|\leq \delta \\ s,t \in [0,T]}}|X^n_t - X^n_s|>\epsilon 
	\right)=0, \mbox{ for all } \epsilon > 0.
	\end{equation}
\end{proposition}

\bigskip

More precisely, we are going to prove the following:

\begin{theorem}\label{tightness}
	For any $ \lambda \in (1,2) $, there exists a constant $ C = C(\lambda) $ such that
	\begin{equation}\label{perro}
	\EE \lrc{\Big|\int_t^{t+\tau}\omega(\xi^n_s) - v(\rho)
	\Big|^{\lambda}\,ds}
	\leq C\tau^{3\lambda /4}
	\end{equation}
	holds for every $ t,\tau \in [0,T] $ and for every $ n\in\mb N $. In particular, we can take $ \lambda \in (\frac{4}{3}, 2) $ and apply Proposition \ref{kolmogorov_centov} to show that the sequence $A^n$ is tight in $ C_{[0,T]}\mb R $.
\end{theorem}

In the remaining of the section, we are going to prove Theorem \ref{tightness}. The plan is the following: first, we simplify the problem by noticing that we only need to consider additive functionals of the form 

\begin{equation*}
\sqrt n \int_0^t  \xxi^n_s(x_1)\cdots \xxi^n_s(x_k)\,ds;
\end{equation*}
for these functionals, estimate 
\eqref{perro} amounts to a careful reproving of the Replacement Lemma. From a technical point of view, the proofs of the Replacement Lemma and the Entropy Estimate are very similar, hinging upon the estimation of certain time integrals of the process. The estimate is always done in two steps: first one replaces local functions by their space averages and then one makes use of concentration inequalities to bound the averages.

\bigskip

Recall the definition of $A^n$ in \eqref{gato}. Each term inside the time integral is a mean-zero local function. Every such function can be written as a polynomial in the variables $\{\xi_x - \rho\}_{x\in \mc R}$. The number of terms of this polynomial does not depend on $ n $. Therefore, it is enough to prove

\begin{equation}\label{pikachu}
\bb E_n\left[ \Big| \sqrt n \int_t^{t+\tau}  \xxi^n_s(x_1)\cdots \xxi^n_s(x_k)\,ds
\Big|^{\lambda} \right]\leq C\cdot\tau^{3\lambda/4}
\mbox{ for all }t,\tau\leq T.
\end{equation}

To keep notation simple, we are going to prove the special case where the local function is $\xxi_0\xxi_1$. The proof carries almost without modification to more complicated polynomials.

It will be more convenient to work with tail bounds instead of moments. The following lemma will help us to connect tail bounds and moment estimates. Its proof is in the Appendix.

\begin{lemma}\label{tail_lemma}
	Let $ X $ be a nonnegative random variable. Assume $ \mb P(X>\delta) \leq C/\delta^2$ for any $ \delta > 0 $. Then, for any $ \lambda \in (1,2) $, there exists an universal constant $ C(\lambda) $ such that $ \mb E[X^\lambda] \leq C(\lambda)\cdot C^{\lambda/2} $.
\end{lemma}

\bigskip

\noindent\textbf{Step 1: Concentration} For a given $\ell \in \bb N$, denote  $\overline{\xi^{\ell}_s}:=\frac{1}{\ell}(\xxi^n_s(-1)+\cdots + \xxi^n_s(-\ell))$. Then, for $\ell = \lfloor n \sqrt \tau \rfloor$ and for $\lambda \in (1,2)$,
\begin{equation}\label{bulbasaur}
\bb E_n\left[\Big|  \sqrt n \int_t^{t+\tau}  (\overline{\xi^{\ell}_s})^2\,ds
 \Big|^{\lambda} \right] \leq C(\lambda)\cdot \tau^{3\lambda/4}.
\end{equation}

\noindent\textbf{Step 2: Replacement} With the same notation as in Step 1, for $\ell = \lfloor  n \sqrt \tau \rfloor$,

\begin{equation}
\bb E_n\left[\Big|  \sqrt n \int_t^{t+\tau}  \xxi^n_s(0)\,\xxi^n_s(1)-(\overline{\xi^{\ell}_s})^2\,ds
 \Big|^{\lambda} \right] \leq C(\lambda)\cdot \tau^{3\lambda/4}.
\end{equation}

\bigskip

\noindent\textbf{Proof of Step 1: } During the proof, $C$ will denote a positive number that may change from line to line. It depends on $\lambda$ and $T$ but not on any other parameter.

Since $|\overline{\xi^\ell_s}|\leq 1$, we can prove \eqref{bulbasaur} with $|\overline{\xi^\ell_s}|$ in place of $(\overline{\xi^\ell_s})^2$. 

By the entropy inequality,

\begin{equation}\label{ivysaur}
\bb P_n\left(  n^{\frac{1}{2}}|\overline{\xi^\ell_s}|
  > \delta  \right)\leq \frac{H_n(s)+\log 2}{\log \left\{ 1+
 \bb P_n\left(n^{\frac{1}{2}}  |\overline{\xi^\ell_0}|
  > \delta  \right)^{-1} \right\}}
\end{equation}

Recall that, under the initial measure $\nu_{\rho}$, the random variables $\{\xxi_x\}_{x\in\bb Z}$
are independent. By Hoeffiding's Inequality, 

\begin{equation}\label{venusaur}
\bb P_n\left(n^{\frac{1}{2}}  |\overline{\xi^\ell_0}|
  > \delta  \right)\leq 2e^{-2\ell\delta^2/n}.
\end{equation}

We have already proved in Section \ref{entropy_section} that $H_n(s)\leq Cs$ for some universal constant $C$. Combining this fact with \eqref{venusaur} and \eqref{ivysaur}, we can prove

\begin{equation}
\begin{aligned}
\bb P_n\left(n^{\frac{1}{2}}  |\overline{\xi^\ell_s}|
  > \delta  \right)
&  \leq
\frac{H_n(s)+\log 2}{\log\lrch{1+\frac{1}{2}e^{2\frac{\ell}{n}\delta^2}  }}  \\
& \leq \frac{2\lrp{H_n(s)+\log 2}}{2\frac{\ell}{n}\delta^2}\\
&\leq C \frac{ n}{\ell}\frac{1}{\delta^2}.
\end{aligned}  
\end{equation}

Applying Lemma \ref{tail_lemma} and recalling our choice $\ell = n\sqrt t$,

\begin{equation}
\bb E_n [|n^{\frac{1}{2}}  \overline{\xi^\ell_s}|^{\lambda}]\leq C/t^{\lambda/4},\mbox{ for all }s\leq T.
\end{equation}

We finish the proof with an application of Jensen's inequality:

\begin{equation}
\begin{split}
\bb E_n [|\int_t^{t+\tau}n^{\frac{1}{2}}  \overline{\xi^\ell_s}\,ds|^{\lambda}] &\leq 
\tau^{\lambda} \cdot \frac{1}{\tau}\int_t^{t+\tau}\bb E_n[|n^{\frac{1}{2}}  \overline{\xi^\ell_s}|^{\lambda}]\,ds\\
&\leq C\cdot \tau^{3\lambda/4}.
\end{split}
\end{equation}

\hfill \qed

%
%

\bigskip 

\noindent\textbf{Proof of Step 2: } Applying Lemma \ref{tail_lemma}, we see that it suffices to prove, for all $\delta >0$ and $\tau \leq T$, 

\begin{equation}
\begin{split}
\bb P_n\left(\Big|  \sqrt n \int_t^{t+\tau}  \xxi^n_s(0)\cdot(\xxi^n_s(1)-\overline{\xi^{\ell}_s})\,ds
 \Big|>\delta \right) & \leq  C \,\tau^{3/2} / \delta^2\mbox{ and }\\
 \bb P_n\left(\Big|  \sqrt n \int_t^{t+\tau}  \overline{\xi^{\ell}_s}\cdot (\xxi^n_s(0)-\overline{\xi^{\ell}_s})\,ds
 \Big|>\delta \right) & \leq C \,\tau^{3/2} / \delta^2.
\end{split}
\end{equation}

We are going to prove the first inequality only, because the second is analogous. The first idea that comes to mind is to adapt the proof of the Replacement Lemma. In trying that, we run intro trouble when trying to control  term  \eqref{adjoint} in the variational problem. Thus we use the following trick: first, we subtract the troubling term in advance, so that it does not show up in the variational problem; then we estimate it by a separate argument, taking advantage of the entropy estimate proven in Theorem \eqref{entropy_bound}. The proof follows from the following Lemma.

\begin{lemma}\label{removing_annoying_term}
Recall the notation \eqref{charmander}. There exists $\theta_0>0$ such that 

\begin{equation}\label{bellsprout}
\log \bb P_n\left( \int_t^{t+\tau} \pm n^{\frac{1}{2}}\,  \xxi^n_r(0)\cdot(\xxi^n_r(1)-\overline{\xi^{\ell}_r}) - \theta_0 n\,\psi(\xi^n_r)\,dr
>\delta \right)  \leq - C \delta^2/ \tau^{3/2}.
\end{equation}
In fact, we can take $\theta_0=2\tau^{3/2}/\delta\epsilon_0$. The same $ \theta_0 $ satisfies

\begin{equation}\label{putting_annoying_term_back}
\bb P_n\left(\Big|
\int_t^{t+\tau} \theta_0\,n\,\psi(\xi^n_r)\,dr
\Big|>\delta \right)
\leq C \tau^{3/2}/\delta^2.
\end{equation}
\end{lemma}

\noindent\emph{Proof of Lemma \ref{removing_annoying_term}: }
Let $\theta >0$. Apply three inequalities: first, Markov's inequality, $\bb P[X> \delta]\leq e^{-\theta\delta} \bb E[e^{\theta X}]$; second, Feynman-Kac's inequality, Proposition \ref{Feynman-Kac}; finally, the bound \eqref{squirtle}. We conclude that \eqref{bellsprout} is bounded by

\begin{equation}\label{oddish}
-\theta \delta + \tau\cdot \sup_f\left\{\lrb{\,
\xxi_0(\xxi_1 - \xxi^\ell), \theta \sqrt n f} -\epsilon_0 n^2 \mc D(\sqrt f) + (1 - \theta \theta_0)\lrb{\psi,nf}
\right\},
\end{equation}
where the supremum is taken over the set of probability densities with respect to $\nu_{\rho}$. Recall that $\xxi^{\ell} = \frac{1}{\ell}(\xxi_{-1}+\cdots + \xxi_{-\ell})$. 
Writing
\begin{equation}
\begin{split}
\xxi_1 - \xxi^\ell &= (\xi_1 - \xi_0) + (\xi_0 - \xi_{-1})+ \frac{\ell - 1}{\ell}
(\xi_{-1}-\xi_{-2})+\cdots + \frac{1}{\ell}(\xi_{-\ell + 1}-\xi_{-\ell}) \\
\end{split}
\end{equation}
and applying Lemma \ref{l1}, we find that, for any $\gamma >0$, 
\begin{equation}\label{weepinbell}
\lrb{\,
\xxi_0(\xxi_1 - \xxi^\ell), \theta \sqrt n f} \leq \gamma\mc D(\sqrt f)+ \frac{n \ell\, \theta^2}{\gamma}.
\end{equation}

%
%
%

Going back to \eqref{oddish}, choose $\gamma=\epsilon_0 n^2$ and $\theta_0 = \theta^{-1}$. Recall that $\ell = n\sqrt \tau$. Then \eqref{oddish} is bounded by $-\theta\delta + \frac{\theta^2\tau^{3/2}}{\epsilon_0}$. We can choose $\theta = \delta\epsilon_0/2\tau^{3/2}$. This proves the first inequality in the statement of Lemma \ref{putting_annoying_term_back}.

\bigskip

With this choice of $ \theta_0 $, the second inequality can be written as 

\begin{equation}\label{key}
\PP_n\lrp{\Big|
\int_t^{t+\tau}	n\psi(\xi^n_s)\,ds
\Big|>\delta} 
\leq
\frac{C\theta_0}{\delta }, 
\end{equation}
and for that it is enough to show
\begin{equation}
\bb E_n\lrc{\Big|\int_t^{t+\tau} n\,\psi(\xi^n_s)\,ds\Big|} = O(1).
\end{equation}

The entropy inequality gives the bound

\begin{equation}\label{charizard}
H_n(t) + \log\bb E\left[\exp\Big|\int_0^{\tau} n\,\psi(\xi^n_s)\,ds
\Big|
\right].
\end{equation}
This estimate is implicit in the proof of the entropy bound. Apply four inequalities in sucession; first, the entropy bound (Theorem \ref{t3}); second, $e^{|a|}\leq e^a + e^{-a}$; third, Feynman-Kac's inequality (Propostion \ref{Feynman-Kac}); finally, Lemma \ref{lema_fundamental}.

\section{Limit Points of the Additive Functional}\label{sec_characterization}

In the previous section we proved that the sequence of additive
functionals 
\begin{equation}\label{key}
\lrch{A^n_t:=\int_0^t\sqrt n (\omega(\xi^n_s)-v(\rho))\,ds:t\in [0,T]}_{n\in\NN}
\end{equation}
is tight. In this section we identify
its limit points, in Proposition \ref{characterization}. For that we will rely strongly on the results of \cite{slafips}.

By the Replacement Lemma \ref{t4} we can approximate $ A^n_t $ by the additive functional $ v'(\rho)\sqrt{n}\int_0^t \xi_s^n \star \varphi_\epsilon\,ds $. Following \cite{slafips}, we relate this functional to the density fluctuation field of the underlying particle system. 
One can write 
\begin{equation}\label{key}
\sqrt n\int_0^t \xi_s^n \star \varphi_\epsilon\,ds = \int_0^t X^n_s\lrp{
	\tau_{-x^n_s/n}\varphi_\epsilon}\,ds.
\end{equation}
By Theorem \ref{lln}, the rescaled random walk $\frac{x^n_s}{n}$ converges to a deterministic trajectory, $ \{v(\rho)t:t\in[0,T]\} $. Because of that, we expect the integral $\int_0^t\sqrt{n}(\omega(\xi^n_s)-v(\rho))\,ds$ to behave like $v'(\rho)\int_0^t X^n_s(\tau_{-v(\rho)s}\varphi_\epsilon)\,ds$. The scaling limit of this last process is defined using the following result:

\begin{theorem}\label{zepsilon}
	Let $ \mc Y_t $ denote the stationary solution of the Ornstein-Uhlenbeck equation with drift $ a \neq 0 $:
	
	\begin{equation}\label{ou_with_drift}
	d\mc Y_t:=D(\rho)\Delta \mc Y_t\,dt + a\nabla \mc Y_t\,dt + \sqrt{2D(\rho)\chi(\rho)}d\nabla \dot{\mc W_t}.
	\end{equation}
	For $ \epsilon\in (0,1) $, let $ i_\epsilon(u) =\epsilon^{-1}\ind{(0,\epsilon]} $ and let $\{ \mathcal Z^\epsilon_t:t\in [0,T]\} $ be the process defined by 
	\begin{equation}\label{occupation_origin}
	\mathcal Z^\epsilon_t:=\int_0^t \mc Y_s(i_\epsilon)\,ds.
	\end{equation}
	Then the sequence of processes $ \{\mathcal Z^{\epsilon} \}_{\epsilon > 0} $ converges in the uniform topology of $ C([0,T];\RR) $ to a  Gaussian process $ \{\mathcal Z_t:t\in [0,T]\} $ of stationary increments and variance
	\begin{equation}\label{key}
	\EE[\mathcal Z_t^2]= D(\rho)\chi(\rho)\sqrt{\frac{2}{\pi}}\int_0^t\frac{(t-s)e^{-\frac{a^2}{2}s}}{\sqrt s}\,ds.
	\end{equation}
	
	The same statement holds if $ i_\epsilon $ is replaced by a smooth function $ \varphi_\epsilon $ with support contained in $ (0,\epsilon) $.
	
\end{theorem}

This theorem corresponds to Theorem 6.3 of \cite{slafips}. The extension to general approximations of the identity $\varphi_\epsilon$ is trivial, so we do not discuss it here.

Now we have all the definitions needed to characterize the limit points of the additive funtional $ A^n_t $.

\begin{proposition}\label{characterization}
	Let $\{A_t:t\in [0,T]\}$ be a limit point of the sequence $A^n$, defined in \eqref{gato}. Let $\mathcal Z^\epsilon$, $\mathcal Z$ be the processes defined in Theorem \ref{zepsilon}, with $a=v(\rho)$. Then $A$ and $v'(\rho)\mathcal Z$ have the same finite-dimensional distributions.
\end{proposition}

We begin with a lemma that allows us to write $\mathcal Z^\epsilon$ in a more convenient way.

\begin{lemma}\label{moving_fluct_field}
	Let $ X $ be the stationary solution of the Ornstein-Uhlenbeck equation \eqref{ou} and $ v(\rho) $ be as in \eqref{asymptotic_speed}. Denote $ \tau_x f(u):=f(x+u) $. Then the process $\{\mc Y_t:t\in [0,T]\}$ defined by 
	\begin{equation*}
	\mc Y_t(f):=X_t(\tau_{-v(\rho)t}f)
	\end{equation*}
	is a solution of the Ornstein-Uhlenbeck equation with drift \eqref{ou_with_drift}, with drift $ a= v(\rho) $.
\end{lemma}

\begin{proof}
	The proof is a simple computation. We would like to show that, for any sufficiently smooth $H: [0,T]\times \RR \to \RR$, the process $ \{M_t(H): t\in [0,T]\}$ defined by
	
	\begin{equation*}
	M_t(H):=\mc Y_t(H_t)- \mc Y_0(H_0) - \int_0^t\mc Y_s((\partial_s + D(\rho)\Delta + v(\rho)\nabla)H_s)\,ds
	\end{equation*}
	is a martingale with quadratic variation 
	\begin{equation*}
	\lrch{\int_0^t 2 D(\rho) \chi(\rho) ||\nabla H_s||^2_{L^2(\RR)}\,ds:t\in[0,T]}.
	\end{equation*}
	Substituting the definition of $ \mc Y $ in the formula for the martingale, we find
	
	\begin{equation*}
	M_t(H)=X_t(\tau_{-v(\rho)t}H_t)-X_0(H_0)-\int_0^t X_s((\partial_s +D(\rho) \Delta)\tau_{-v(\rho)s}H_s)\,ds.
	\end{equation*}

	Since $X$ solves the Ornstein-Uhlenbeck equation without drift \eqref{ou}, the expression above is a martingale with quadratic variation
	
	\begin{equation*}
	\begin{aligned}
	\lrb{ M_t(H)} &= \int_0^t 2D(\rho)\chi(\rho) ||\nabla(\tau_{-v(\rho)s}H_s)||^2_{L^2(\RR)}\,ds \\
	&= \int_0^t2D(\rho)\chi(\rho) ||\nabla(H_s)||^2_{L^2(\RR)}\,ds,
	\end{aligned}
	\end{equation*}
	as we wanted to show.
\end{proof}

\begin{proof}[Proof of Proposition \ref{characterization}]
	Let $X^n$ denote the density fluctuation field associated to the lattice-gas process, defined in \eqref{fluctuation_field}, and $X$ its limit. For each $ \epsilon > 0 $, let $ \varphi_\epsilon $ be a nonnegative smooth function with support in $ (0,\epsilon) $.   Consider the auxiliary processes $\{A^{n,\epsilon}_t:t\in[0,T]\}$  defined by
	
	\begin{equation}\label{averaged_additive_functional}
	A^{n,\epsilon}_t :=v'(\rho)\int_0^tX^n_s(\tau_{-x^n_s/n}\varphi_\epsilon)\,ds.
	\end{equation}
	
	First, we claim that $ A^{n,\epsilon} $ converges, as first $ n\to\infty $ then $ \epsilon\to 0 $, to $v'(\rho) \mathcal Z $, in the sense of finite-dimensional distributions. To prove this claim, we 
	put together  Proposition \ref{p2} (fluctuations of the lattice-gas) and Theorem \ref{lln} (Law of Large Numbers for the random walk) and see that $ A^{n,\epsilon} $ converges, as $ n\to \infty $, to the process
	
	\begin{equation}\label{occupation_rw}
	\lrch{
		v'(\rho)\mathcal Z^{\epsilon}_t :=v'(\rho)\int_0^{t}X_s\lrp{\tau_{-v(\rho)s}\varphi_{\epsilon}}\,ds: t\in [0,T]}
	\end{equation}
	in the sense of finite-dimensional distributions. Using Lemma \ref{moving_fluct_field} and Theorem \ref{zepsilon}, we see that 
	$ \mathcal Z^{\epsilon} $ converges, as $ \epsilon\to 0 $, to $ \mathcal Z $. 
	More details for the first assertion:
	consider, for each $t\in[0,T]$, the function $  F_t:\mc D([0,T];H_{-2}(\RR))\times \mc D([0,T];\RR) \to \RR$ given by $F_t(\mathcal X, x):= \int_0^t \mathcal X_s(\tau_{x_s}\varphi_\epsilon)\,ds$. 
	Notice that all trajectories in 
	$\mc C([0,T]; H_{-2}(\RR))\times \mc C ([0,T];\RR) $ are 
	continuity points for $F_t$; it is in order to ensure this continuity that we average with a smooth approximation of the identity. We can 
	write $A^{n,\epsilon}_t= v'(\rho)F_t(X^n,-x^n/n)$. 
	Using the continuity of $F_t$ and the 
	convergences of $X^n$ and $x^n/n$ we conclude 
	that $A^{n,\epsilon}_t$ converges weakly to 
	$\mathcal Z^\epsilon_t$. In the same 
	way, we can prove convergence of arbitrary
	finite-dimensional distributions.

	Second, we claim that the processes $ A^n $ and $ A^{n,\epsilon} $ have the same limit points, as first $ n\to \infty $ then $ \epsilon\to 0 $. This follows from Theorem \ref{t4}. In more detail:
	fixing $\epsilon >0$ and $t>0$, the sequence $(A^n_t, A^{n,\epsilon}_t)_{n\in\NN}$ is tight in $\RR^2$ and its limit points
	$(A_t, v'(\rho)\mathcal Z^\epsilon_t)$ satisfy $\EE|A_t-v'(\rho)\mathcal{Z}^\epsilon_t|\leq \varlimsup_{n}\EE|A^n_t - A^{n,\epsilon}_t|$. In the same way, the family $(A_t, \mathcal{Z}^\epsilon_t)_{\epsilon > 0}$ is tight in $\RR^2$, and its limit points $(A_t, v'(\rho)\mathcal Z_t)$ satisfy $\EE|A_t - v'(\rho)\mathcal Z_t|\leq \varlimsup_\epsilon |A_t - v'(\rho)\mathcal Z^\epsilon_t|$. By Theorem \ref{t4}, $\EE|A_t - v'(\rho)\mathcal Z_t|=0$. This shows that the processes $A$ and $v'(\rho)\mathcal Z$ have the same marginals. An analogous (but notationally more cumbersome) argument takes care of arbitrary finite-dimensional distributions.
\end{proof}

\section{Asymptotic Independence}

Our starting point in the study of the random walk was to write down the decomposition \eqref{decomposition}.  The position of the (centered and scaled) random walk, $\frac{x^n_t - v(\rho)nt}{\sqrt{n}}$, is a sum of a martingale $M^n_t$ and an additive functional $A^n_t$. We proved that the martingale part converges to Brownian motion and that the additive functional converges to a Gaussian process with stationary increments. In this section we show that these limiting processes are independent, or, putting it more precisely, that the sequence of random vectors $(M^n,A^n)$ converges in law to a product measure on $(\mathcal{C}([0,T], \RR))^2$. 

First we tackle the problem of proving that $M_t$ is independent of $A_t$ for each $t\in [0,T]$. In view of the Replacement Lemma \ref{t4} and the Law of Large Numbers \ref{lln} , we can try to approximate $A_t$ by the additive functional
\begin{equation}\label{asymptotic_additive_functional}
v'(\rho)\sqrt n \int_0^t(\eta_s^n \star \varphi_\epsilon)(v(\rho)s)\,ds.
\end{equation}

The functional \eqref{asymptotic_additive_functional} depends only on the environment. Our strategy to prove asymptotic independence of the processes $ M $ and $ A $ is to construct a martingale $( N^n_{s,t})_{ s\leq t}$ such that $N^n_{t,t}$ approximates the integral in \eqref{asymptotic_additive_functional}. This martingale will be a function of the environment process alone, and therefore it will never jump at the same time as the walker. Besides, $M^n$ jumps only when the walker jumps, so the martingales $( M^n_{s})_{ s\leq t}$ and $( N^n_{s,t})_{ s\leq t}$ will be orthogonal. If in addition the quadratic variation $(\lrb{N^n_{s,t}})_{s \leq t}$ converges to an increasing function of $ s $, we can apply the Martingale FCLT to conclude that $\{(M^n_s, N^n_{s,t}):s\leq t\}_{n\in\NN}$ converges to a pair of independent continuous martingales $M$ and $N$. In particular, $M_t$ is independent of $N_{t,t}=A_t$.

\begin{lemma}\label{independent_marginals}
	Let $(M, A)$ be a limit point of the sequence $(M^n,A^n)$ and $t\in [0,T]$. Then $M_t$ is independent of $A_t$.
\end{lemma}

\begin{proof}
	%
	Let $ \epsilon >0 $. Recall the definition of the additive functional $ A^{n,\epsilon} $ in \eqref{averaged_additive_functional}. For the purposes of the present lemma, we can assume that $ v'(\rho)=1 $ without loss of generality. To build a martingale that is close to $ A^{n,\epsilon} $ at time $ t $, we write down a Dynkin martingale with the following test function: let $H^\epsilon:[0,t]\times \RR$ be the solution of 
	
	\begin{equation*}
	\left\{
	\begin{array}{rlll}
	\partial_sH^\epsilon(s,u)+\partial_{uu}H^\epsilon(s,u) &=& \varphi_\epsilon(v(\rho)s + u) & \mbox{ for all } s\in [0,t], u\in\RR \\
	H^\epsilon(t,u) &=& 0 & \mbox{ for all }u\in\RR.
	\end{array}
	\right.
	\end{equation*}
	
	Let $\Delta_n f(u):=n^2[f(u+n^{-1})+f(u-n^{-1})-2f(u)]$.  Then 
	
	\begin{equation}\label{bivalued_martingale}
	N^{n,\epsilon}_{s,t}:=-\YY^n_s(H^\epsilon_s)+\int_0^s \YY^n_r((\partial_r+\Delta_n)H^\epsilon_r)\,dr,
	\end{equation}
	defined for $ s\leq t $, is a martingale with quadratic variation
	
	\begin{equation*}
	\lrb{N^{n,\epsilon}_{s,t}}= \int_0^s \frac{1}{n}\sum_{x\in\ZZ} n^2 \left(H^\epsilon_r\left(\frac{x+1}{n}\right)-H^\epsilon_r\left(\frac{x}{n}\right)
	\right)^2(\eta^n_r(x+1)-\eta^n_r(x))^2\,dr.
	\end{equation*}
	
	Notice that the above martingale does not start at zero. 	
	By the Law of Large Numbers \eqref{lln}, the difference $ N^{n,\epsilon}_{t,t} - A^{n,\epsilon}_t $ converges to zero in probability, as $ n\to\infty $. Here is one of the points in the proof where we need the smoothness of $\varphi_\epsilon$, for this ensures smoothness of $H^\epsilon$, and therefore a $O(n^{-1})$ error in the approximation of $\partial_{uu} H^\epsilon$ by $\Delta_n$. The proof of the lemma will be complete when we show that the sequence of random vectors $ \{ (M^n_t, N^{n,\epsilon}_{t,t}):n\in \NN, \epsilon >0 \} $ converges in law, as first $ n\to \infty $ then $ \epsilon \to 0 $, to a random vector of independent marginals.
	
	With this goal in mind and looking for an opportunity to apply the MFCLT, we claim that
	\begin{equation*}
	\lim_{n\to\infty} \lrb{ N^{n,\epsilon}_{s,t} } = \int_0^s 2\rho(1-\rho)||\partial_u H^\epsilon_r(u)||^2_{L^2(\RR)}\,dr \mbox{   in probability.}
	\end{equation*}
	
	To compute this limit, we start by setting 
	\begin{equation}\label{auxiliary_integral}
	f_r(x,n):= \int_{x/n}^{(x+1)/n}(\partial_u H^\epsilon_r(u))^2\,du.
	\end{equation}
	Using a Taylor expansion for $ H_r $ and the fact that $ H_r $ has compact support, it is possible to show that $ \lrb{ N^{n,\epsilon}_{s,t} } $ has the same limit as
	\begin{equation}\label{key}
	\int_0^s \sum_{x\in \ZZ} f_r(x,n)(\eta^n_r(x+1)-\eta^n_r(x))^2\,dr.
	\end{equation}
	To replace $(\eta^n_r(x+1)-\eta^n_r(x))^2$ by its mean $2\rho(1-\rho)$, one can explore the fact that if a sequence of random variables $ X_n $ satisfies $ \EE X_n \to 0 $ and $ \mbox{Var}X_n \to 0$ then $ X_n \to 0 $ in probability. To estimate the variance, we use Cauchy-Schwarz inequality and stationarity:
	
	\begin{equation}\label{variance_quad_var}
	\begin{aligned}
	&\EE_n \lrc{\lrp{\int_0^s
			f_r(x,n)(\eta^n_r(x+1)-\eta^n_r(x))^2\,dr}^2
	} \\
	\leq\quad & 2\rho(1-\rho) s\cdot \int_0^s \sum_{x\in\ZZ} f_r(x-1,n)f_r(x,n) + f_r(x+1,n)f_r(x,n)\,dr.
	\end{aligned}
	\end{equation}
	Recall that $ H^\epsilon_r $ has compact support for all $ r \in [0,t]$, so that only a finite number of terms in the sum above are not null. From the definition of $ f_r $, \eqref{auxiliary_integral}, we see that each $ f_r(x,n) $ is of order $ n^{-1} $, so that the variance in \eqref{variance_quad_var} does indeed converge to $ 0 $ as $ n\to\infty $.
	
	Using the Martingale FCLT, we see that $ \{(M^n_s, N^{n,\epsilon}_{s,t}):s\leq t\}_{n\in\NN} $ converges to a continuous Gaussian process $ \{(M_s, N^\epsilon_{s,t}):s\leq t\} $ with independent increments. Since $ M^n $ is orthogonal to $ N^{n,\epsilon} $, it follows that the limit has independent marginals, as we wanted to show. 
	
	
\end{proof}

Let $ 0\leq t_1<\cdots < t_k $. We finish the section by indicating how to prove that the finite-dimensional distributions $(M_{t_1},\ldots, M_{t_k})$ and $(A_{t_1},\ldots, A_{t_k})$ are independent. The proof builds upon the strategy used in Lemma \ref{independent_marginals}.

\begin{theorem}
	Let $(M, A)$ be a limit point of the sequence $(M^n,A^n)$. Let $0<t_1<\cdots <t_k \leq T$. Then  $(M_{t_1},\ldots, M_{t_k})$ and $(A_{t_1},\ldots, A_{t_k})$ are independent.
\end{theorem}

\begin{proof}
	To simplify the notation, let us do the case with just two times $s$ and $t$, with $s<t$.  Assume also that $ v'(\rho)=1 $, without loss of generality. It is enough to show that $(M_s,M_t)$ is independent of $\mathcal (Z^\epsilon_s,\mathcal Z^\epsilon_t)$, for each $ \epsilon > 0 $.
	
	\noindent \textbf{Step 1: }Using characteristic functions, we see that it suffices to prove that $a_1M_s+a_2M_t$ is independent of $b_1 \mathcal Z^\epsilon_s + b_2 \mathcal Z^\epsilon _t$ for any $a_1, a_2, b_1, b_2 \in \RR$. 
	
	\noindent \textbf{Step 2: }Define the Dynkin martingales $\{N^n_{r,s}: r\leq s\}$ and $\{N^n_{r,t}:r\leq t\}$ as in \eqref{bivalued_martingale} (notice that the test function $H^\epsilon$ in used in \eqref{bivalued_martingale} depends on $t$). 
	
	Notice that 
	\begin{equation}
	\lrch{b_1N^n_{r\wedge s,s}+b_2N^n_{r,t} :r\leq t }
	\end{equation}
	is also a Dynkin martingale. One can show, repeating the proof in Lemma \ref{independent_marginals},  that its quadratic variation converges,  as $n\to\infty$, to an increasing function of $r$.
	
	\noindent \textbf{Step 3: } Using the Martingale FCLT, we see that the sequence
	\begin{equation}
	\{(M^n_r,b_1N^n_{r,s}+b_2N^n_{r,t}):r\leq t\}_{n\in\NN}
	\end{equation}
	converges weakly and the limit has independent marginals. In particular, $b_1\mathcal Z^\epsilon_s + b_2 \mathcal Z^\epsilon_t = b_1N^n_{t,s}+b_2N^n_{t,t}$ is independent of $M$.
\end{proof}

\appendix

\section{Variational inequalities}

In this section we prove variational inequalities relating the Dirichlet form $\mc D(\sqrt f)$ with various integrals of interest. We start with some definitions. Recall the definition of the Dirichlet form:
\[
\mc D(\sqrt f) = \sum_{x \in \bb Z} \int \big(\sqrt {f^{x,x+1}}-\sqrt{f}\big)^2\,d\nu_{\rho}
\]

We have the following result:

\begin{lemma}
	\label{l1}
	Let $f$ be a density with respect to $\nu_{\rho}$, that is, $f\geq 0$ and $\int f\,d\nu_{\rho}=1$. Fix $x\in \bb Z$ and $ \beta >0 $. Let $ g $ be a local function and let $ h $ be a bounded function such that $ h(\eta^{x,x+1})=h(\eta) $ for all $ \eta \in \Omega $. Then
	
	\begin{equation}\label{by_parts}
	\lrb{fh,g^{x,x+1}-g}\leq \beta\mc D^{x,x+1}(\sqrt f) + \frac{1}{\beta}\lrb{g^2 + (g^{x,x+1})^2,fh^2}.
	\end{equation}
	
\end{lemma}

\begin{proof}
	Since $\nu_\rho$ is invariant with respect to the change of variables $\xi\mapsto \xi^{x,x+1}$, we have
	
	\begin{equation}
	\lrb{f, g^{x,x+1}-g} = \frac{1}{2}\lrb{f - f^{x,x+1},g^{x,x+1}-g}.
	\end{equation}
	
	Write $A = \frac{1}{2}h(g^{x,x+1}-g)$, $B = f$ and $C = f^{x,x+1}$. We have that
	\[
	B - C = \big( \sqrt{B} -\sqrt{C}\big) \big(\sqrt{B} +\sqrt{C}\big),
	\]
	and using the weighted Cauchy-Schwartz inequality we get 
	\[
	A\big( B-C\big) \leq \beta  \big( \sqrt{B} -\sqrt{C}\big)^2 + \frac{A^2 \big(\sqrt{B} +\sqrt{C}\big)^2}{4\beta}.
	\]
	Notice that $(\sqrt{B}+\sqrt{C})^2 \leq 2(B+C)$, whence
	\begin{equation}
	\label{gallo}
	A\big( B-C\big) \leq \beta  \big( \sqrt{B} -\sqrt{C}\big)^2 + \frac{A^2 (B +C\big)}{2\beta}.
	\end{equation}
	Recall the definitions of $A,B$ and $C$. We have that $A^2 \leq h^2(g^2+(g^{x,x+1})^2)$. Integrating \eqref{gallo} with respect to $\nu_\rho$ we obtain the lemma.
\end{proof}


Inequality \eqref{by_parts} will be very helpful to get bounds on the exponential moments of the additive functionals of $\{\xi_t^n:t\geq 0\}$. More precisely, combining Theorem A1.7.2 and equation A3.1.1 of  \cite{kl}, we have the following:
\begin{proposition}
	\label{Feynman-Kac}
	Let $V: \Omega \to \bb R$ be a bounded function. Then,
	\[
	\log \bb E_n \lrc{ e^{\int_0^t V(\xi_s^n) ds}} \leq t \sup_{f} \lrch{ \lrb{V,f} + \lrb{\sqrt f, L_n \sqrt f}},
	\]
	where the supremum is over all densities $f$ with respect to $\nu_\rho$.
\end{proposition}

Combining this proposition with Lemma \ref{lema_fundamental}, we get the following estimate:

\begin{lemma}
	\label{l2}
	Let $V: \Omega \to \bb R$ be a bounded function. Then there exists $ C>0 $ that does not depend on $ n $ such that, for all $ \beta >0 $,
	\begin{equation}
	\label{pera}
	\log \bb E_n\big[e^{ \int_0^t V(\xi_s^n) ds }\big] \leq t \sup_{f} \lrch{\lrb{V,f} +(\beta - \epsilon_0) n^2 \mc D\big(\sqrt f\big) + \frac{C}{\beta}}.
	\end{equation}
\end{lemma}

\begin{lemma}[Hoeffding's Inequality]
	Let $ X $ be a mean-zero random variable taking values in the interval $ [a,b] $. Then
	\begin{equation}\label{key}
	\EE\lrc{e^{\theta X}}\leq e^{\frac{\theta^2 \sigma^2}{8}}.
	\end{equation}

\end{lemma}

\begin{lemma}[Subgaussianity]\label{subgaussianity}
	Let $ X $ be a random variable. If 
	\begin{equation}\label{key}
	\EE\lrc{e^{\theta X}}\leq e^{\frac{\theta^2\sigma^2}{2}}\mbox{ for all } \theta >0  
	\end{equation}
then
	\begin{equation}\label{key}
	\log \EE\lrc{e^{cX^2}}\leq 2c\sigma^2\mbox{ for all }c\in (0, \frac{\sigma^2}{2}).
	\end{equation}
\end{lemma}

%
%
%
%
%
%
%
%
%
%
%
%
%

\end{document}